\documentclass[11pt,reqno]{amsart}\usepackage[letterpaper, portrait, margin=1in, headheight=0pt]{geometry}
\usepackage{mathtools}
\usepackage{enumerate}
\usepackage{amsmath,amsthm,amssymb,amsfonts}
\usepackage{hyperref}
\usepackage[normalem]{ulem}
\usepackage{soul}
\usepackage{xcolor}
\usepackage{tikz}
\usepackage{tikz-cd}
\usetikzlibrary{calc, positioning, fit, shapes.misc}
\usepackage{subcaption}
\usepackage{blkarray}
\usetikzlibrary{bending}
\usepackage{cleveref}
\newtheorem{thm}{Theorem}[section]

\newtheorem{cor}[thm]{Corollary}
\newtheorem{conj}[thm]{Conjecture}
\newtheorem{prop}[thm]{Proposition}

\theoremstyle{definition}

\newtheorem{exam}[thm]{Example}
\newtheorem{defi}[thm]{Definition}

\title{An Alon-Tarsi Style Theorem for Additive Colorings}

\begin{document}

\setstcolor{red}
\maketitle
\begin{center} {Ian Gossett\

Department of Mathematics and Computer Science, Wesleyan University\

Middletown CT, 06459\ 

igossett@wesleyan.edu}
\end{center}
\begin{abstract}

        We first give a proof of the Alon-Tarsi list coloring theorem that differs from Alon and Tarsi's original. We use the ideas from this proof to obtain the following result, which is an additive coloring analog of the Alon-Tarsi Theorem: Let $G$ be a graph and let $D$ be an orientation of $G$. We introduce a new digraph $\mathcal{W}(D)$, such that if the out-degree in $D$ of each vertex $v$ is $d_v$, and if the number of Eulerian subdigraphs of $\mathcal{W}(D)$ with an even number of edges differs from the number of Eulerian subdigraphs of $\mathcal{W}(D)$ with an odd number of edges, then for any assignment of lists $L(v)$ of $d_v+1$ positive integers to the vertices of $G$, there is an additive coloring of $G$ assigning to each vertex $v$ an element from $L(v)$. As an application, we prove an additive list coloring result for tripartite graphs $G$ such that one of the color classes of $G$ contains only vertices whose neighborhoods are complete.
\end{abstract}

\section{Introduction}
\vspace{.2cm}
An \textit{additive coloring} of a graph $G$ is a function $\ell:V(G)\rightarrow \mathbb{N}$, such that $c:V(G)\rightarrow \mathbb{N}$ defined by $c(v)=\sum_{u\in N(v)}\ell(u)$ is a proper coloring of $G$. Additive colorings, originally called \textit{lucky labelings}, were introduced by Czerwiński et al. in \cite{ref8}. Most of the work pertaining to additive colorings has been devoted to computing the \textit{additive chromatic number}, $\eta (G)$, of specific graphs, where $\eta(G)$ is defined to be the least integer $k$ such that $G$ can be additively colored with elements from the set $\{1,...,k\}$. In this work, we focus on additive colorings where each vertex is assigned its own list of permissible labels.
 
  The\textit{ additive list chromatic number} of $G$, denoted $\eta_l(G)$, is the least integer $k$ such that for any lists of positive integers of length $k$ that are assigned to the vertices of $G$, there is an additive coloring that assigns to each vertex an element from its list. For arbitrary graphs $G$, it was shown in \cite{ref2} that $\eta_\ell (G) \leq \Delta^2-\Delta+1$, where $\Delta\geq 2$ is the maximum degree of $G$. It was shown later in \cite{ref5} that  $\eta_{\ell}(G)\leq k \Delta +1 $, where $k$ is the degeneracy of $G$. For planar graphs, the bound given in \cite{ref5} leads to the upper bound of $\eta_{\ell}(G)\leq 5\Delta +1$, and another linear bound $\eta_\ell (G)\leq 2\Delta +25$ was proved in \cite{ref9}. Some constant bounds for planar graphs of given girth can also be found in \cite{ref7}.  Our main result, Theorem \ref{thm:newat}, may have the potential to improve some of these bounds.

We first present a proof of the Alon-Tarsi list coloring theorem that differs from the original proof given in \cite{ref4}. We use the ideas from this proof to obtain an additive coloring analog of the Alon-Tarsi Theorem: Given an orientation $D$ of $G$, we introduce a new digraph, denoted $\mathcal{W}(D)$, such that if the number of even spanning Eulerian subdigraphs of $\mathcal{W}(D)$ differs from the number of odd spanning Eulerian subdigraphs, then $\eta_l(G)$ is at most one greater than the maximum out-degree of $D$. 

 This paper proceeds as follows: In section 2, we establish basic definitions and background and include our proof of the Alon-Tarsi Theorem. In section 3, we introduce and give examples of the $\mathcal{W}(D)$ construction. In section 4, we prove our main result, Theorem \ref{thm:newat}. In section 5, we give an application of Theorem \ref{thm:newat} to a special class of tripartite graphs. In section 6, we give some concluding remarks and conjecture that for every graph $G$, there is an orientation $D$ of $G$ such that the number of even Eulerian subdigraphs of $\mathcal{W}(D)$ differs from the number of odd Eulerian subdigraphs. 
  \vspace{.2cm}

\section{Definitions and Background}
We establish some notational conventions and definitions. Let $D$ be a digraph. For each $v\in V(D)$, let $d^+_D(v)$ denote the out-degree of $v$ in $D$, and let $d^-_D(v)$ denote the in-degree of $v$ in $D$. For each $v\in V(D)$, let $N_D(v)$ denote the set of vertices adjacent to $v$ in $D$ and let $N_D[v]=N_D(v)\cup \{v\}$. (Define $N_G(v)$ and $N_G[v]$ similarly for undirected graphs $G$.) All graphs and digraphs are assumed to be finite and simple. For sets $U$ and $V$, define $U\triangle V=(U\setminus V)\cup (V\setminus U)$. Let $\mathbb{N}$ denote the set of positive integers. \vspace{.2cm}

\begin{defi}
An \textbf{additive coloring} of a graph $G$ is a function $\ell:V(G)\rightarrow \mathbb{N}$, such that the map $c:V(G)\rightarrow \mathbb{N}$ defined by $c(v)=\sum_{u\in N_G(v)}\ell(u)$ is a proper coloring of $G$; that is, $c(u)\neq c(v)$ for whenever $\{u,v\}\in E(G)$.
\end{defi}

\begin{defi}
Let $D$ be a digraph. $D$ is said to be \textbf{Eulerian} if for each $v\in V(D)$, $d^+_D(v)=d^-_D(v)$. (Note that we do not require Eulerian digraphs to be connected.)
\end{defi}

\begin{defi}
Let $EE(D)$ denote the number of spanning Eulerian subdigraphs of $D$ with an even number of edges, and $EO(D)$ denote the number of spanning Eulerian subdigraphs of $D$ with an odd number of edges. We call an Eulerian subdigraph $\textbf{even}$ if it has an even number of edges, and $\textbf{odd}$ if it has an odd number of edges. 
\end{defi}

 The following theorem is known as the \textit{Combinatorial Nulstellensatz}, and is a main tool used in this paper.

\begin{thm}(Alon, \cite{ref3})
Let $f$ be a polynomial over a field $F$ in the variables $\mathbf{x}=(x_1,x_2,...,x_n)$. Let $d_1,...,d_n$ be nonnegative integers such that the total degree of $f$ is $\sum_{i=1}^{n}d_i$. For each $1\leq i \leq n$, let $L_i$ be a set of $d
_i+1$ elements of $F$. If the coefficient in $f$ of $\prod_{i=1}^{n}x^{d_i}_i$ is nonzero, then there exists $\mathbf{t}\in L_1\times L_2\times \cdots \times L_n$ such that $f(\mathbf{t})\neq 0$.
\end{thm}

\Cref{thm:at} is known as the \textit{Alon-Tarsi Theorem}, and has been used in a variety of list coloring applications (see \cite{ref11} for a nice overview). Our proof of \Cref{thm:at} differs from the proof given in \cite{ref4} in two ways: First, we do not define an adjacency polynomial $f_G$ based on a fixed ordering of the vertices of $G$. Instead, we define the polynomial $f$ so that it depends on a given orientation $D$ of $G$ (and $f$ is equal to $\pm f_G$). Second, we give a direct bijective correspondence between the monomials in the expansion of $f$ and the subdigraphs of $D$, rather than a correspondence that first passes through the set of orientations of $G$. One benefit of this proof is that it can be successfully modified to prove our main result, \Cref{thm:newat}.

\begin{thm} \label{thm:at}(Alon-Tarsi, \cite{ref4}) Let $G=(V,E)$ be a graph and let $D$ be an orientation of $G$ such that  $EE(D)\neq EO(D)$. Let $L(v)$ be a set of $d^+_D(v)+1$ distinct integers for each $v\in V$. Then there is a proper vertex coloring $c:V\rightarrow \mathbb{Z}$ such that $c(v)\in L(v)$ for each $v\in V$. 
\end{thm}

\begin{proof} Identify the vertices of $G$ with the set $\{1,2,..,.n\}$. Note that if there is some $x\in \mathbb{Z}^{|V(G)|}$  so that the polynomial $f\in \mathbb{R}[x_1,x_2,...,x_n]$ defined by $$f(x_1,..,x_n)=\prod_{vw\in \vec{E}(D)}\left(x_v - x_w\right)$$ is nonzero at $x$, then $x$ corresponds to a proper vertex coloring of $G$. We will want to distinguish between the occurrences of a variable $x_v$ inside of a $vw$ factor of $f$ and the occurences $x_v$ inside of other factors. To do this, we temporarily replace the occurences of $x_v$ within a $vw$ factor with a new variable $y_v$; for each $v\in V(D)$, define two variables, $x_v$ and $y_v$, and define the polynomial $\hat{f}\in \mathbb{R}[x_1,...,x_n,y_1,...,y_n]$ by  $$\hat{f}(x_1,...,x_n,y_1,...,y_n)=\prod_{vw\in \vec{E}}\left(y_v-x_w \right).$$

Consider the monomials in the expansion of $\hat{f}$. (Here we are considering the monomials before ``collecting like terms," so there will be $2^{|E(G)|}$ such monomials.) Each monomial corresponds to picking one (signed) variable from each factor of $\hat{f}$. Therefore, if $M$ is such a monomial, we can write $M=\prod_{vw\in \vec{E}(D)}z_{vw}$, where for each $vw\in \vec{E}(D)$,  either $z_{vw}=y_v$, or $z_{vw}=-x_w$. \vspace{.2cm}

Let $M=\prod_{vw\in \vec{E}(D)}z_{vw}$ be a monomial in the expansion of $\hat{f}$, as above. Associate to $M$ the subdigraph $S_M$ of $D$ defined by: $$V(S_M)=V(D)$$
$$\vec{E}(S_M)=\{vw\in \vec{E}(D):z_{vw}=-x_w\}.$$

The following observation is immediate, since an edge $vw$ is included in $S_M$ if and only if $-x_w$ is picked from the $vw$ factor of $\hat{f}$ to contribute to $M$. \vspace{.2cm}

\noindent Observation 1: The correspondence $M\leftrightarrow S_M$ is a bijective correspondence between the spanning subdigraphs of $D$ and the monomials in the expansion of $\hat{f}$.
\vspace{.2cm}

Another immediate observation is that the sign of $M$ is negative in the expansion of $\hat{f}$ if and only if $\#\{vw\in \vec{E}(D):z_{vw}=-x_w\}$ is odd. Therefore, we have: \vspace{.2cm}

\noindent Observation 2: The sign of a monomial $M$ in the expansion of $\hat{f}$ is negative if and only if  $S_M$ has an odd number of edges. \vspace{.2cm}

Fix  $u\in V(D)$.  Let $M=\prod_{vw\in \vec{E}(D)}z_{vw}$ be a monomial in the expansion of $\hat{f}$, as above. For each $v\in V(D)$, denote the degree of $x_v$ in $M$ by $\text{deg}_M(x_v)$. Define the following notation: 
\begin{center}
\begin{align*}
k_u&=\#\{uw\in \vec{E}(D):z_{uw}=-x_w\}\\
t_u&=\#\{vu\in\vec{E}(D):z_{vu}= -x_u\}\\
\end{align*}
\end{center}
 By the construction of $S_M$, there are exactly $k_u$ edges in $S_M$ whose initial vertex is $u$, and exactly $t_u$ edges whose terminal vertex is $u$. Hence, $d^+_{S_M}(u)=k_u,$ and $d^-_{S_M}(u)=t_u$, and we have that $d^+_{S_M}(u)=d^-_{S_M}(u)$ if and only if $t_u=k_u$. Furthermore, note that $t_u=\text{deg}_M(x_u)$,   and since $y_u$ occurs in exactly $d^+_D(u)$ factors of $\hat{f}$,  $\text{deg}_M(y_u)=d^+_D(u)-k_u$.   Therefore, $$d^+_{S_M}(u)=d^-_{S_M}(u) \text{ if and only if }\text{deg}_M(y_u)=d^+_D(u)-t_u=d^+_D(u)-\text{deg}_M(x_u).$$ This yields the following observation: \vspace{.2cm}

 \noindent Observation 3: $S$ is a spanning Eulerian subdigraph of $D$ if and only if $S=S_M$ for some $M$ with $\text{deg}_M(y_v)+\text{deg}_M(x_v)=d^+_D(v)$ for every $v\in V(D)$.  \vspace{.2cm}

If we now let $y_v=x_v$ in $\hat{f}$ for each $v$, we get back the polynomial $f$, and the monomials that are in correspondence with the spanning Eulerian subdigraphs of $D$ are now exactly the monomials $M$ in the expansion of $f$ such that for each $v\in V(D)$, $\text{deg}_M(x_v)=d^+_D(v)$. Write $M_f=\prod_{v\in V(D)}x^{d^+_D(v)}_v$. What we have just determined is that the monomials in the expansion of $f$ that are in correspondence with the spanning Eulerian subdigraphs of $D$ are precisely the occurrences of $\pm M_f$.  Since letting $y_v=x_v$ does not change the sign of any monomial in the expansion, Observations 2 and 3 guarantee that the occurrences of $+M_f$ are in correspondence with even spanning Eulerian subdigraphs, and the occurrences of $-M_f$ are in correspondence with odd spanning Eulerian subdigraphs.

Hence, after collecting like terms, the coefficient of $M_f$ in $f$ is equal to $EE(D)-EO(D)$. Therefore, if $EE(D)\neq EO(D)$, this coefficient is nonzero, and since $M_f$ has maximum total degree in $f$, by the Combinatorial Nullstellensatz, the theorem holds true.

\end{proof}

\section{Constructing $\mathcal{W}(D)$}

In this section, we introduce a new construction that associates to each digraph $D$ another digraph $\mathcal{W}(D)$. This construction is integral to our main result, \Cref{thm:newat}, and indeed, it was devised specifically for the purpose of finding such a theorem. We first give a formal definition of $\mathcal{W}(D)$, followed by an explanation and some examples. We then prove some useful propositions and theorems involving $\mathcal{W}(D)$.\vspace{.2cm}

\begin{defi}
\label{def:WD}
Given a digraph $D=(V,\vec{E})$, we construct a new digraph, $\mathcal{W}(D)$, as follows: 

\begin{enumerate}

\item For each $vw\in \vec{E}(D)$, construct a digraph $H_{vw}$ whose vertex set is given by  $$\{v^{vw}\}\cup \{x^{vw}:x\in N_D(v)\setminus N_D(w) \}$$ 
and whose edge set is  \begin{align*}
\{v^{vw}x^{vw}:x\in N_D(v)\setminus N_D(w) \}.\\
\end{align*}

\item Define the vertex set of $\mathcal{W}(D)$ to be the (disjoint) union $$\left(\bigcup_{vw\in \vec{E}(D)}V(H_{vw})\right)\cup \{x^*:x\in V(D)\},$$

and define the edge set of $\mathcal{W}(D)$ to be \begin{align*} 
&\bigcup_{vw\in \vec{E}(D)}\vec{E}(H_{vw})\\
\cup &\{v^*v^{vw}:vw\in \vec{E}(D)\}\\
\cup &\{x^{vw}x^*:x\in N_D(v)\setminus N_D(w)\}\\ 
\cup &\{v^{vw}x^*: x\in N_D(w)\setminus N_D[v]\}.\\
\end{align*}
\end{enumerate}

\end{defi}

Rephrasing \Cref{def:WD}, we get the following process for constructing $\mathcal{W}(D)$: \\

 Step 1: For each $vw\in \vec{E}(D)$, construct a (star) digraph $H_{vw}$ so that $v^{vw}$ is a source vertex, and if $x\in N_D(v)\setminus N_D(w)$, there is an edge directed from $v^{vw}$ to $x^{vw}$ in $H_{vw}$. 
 
Step 2: Union all of the digraphs $H_{vw}$ from Step 1, along with vertices $v^*$ for each $v\in V(D)$. Orient an edge outward from $v^*$ to each vertex of the form $v^{vw}$ that was a source vertex in step 1.
 
 Step 3: For each vertex $x^{vw}$ from Step 1 with $x\neq v$, orient an edge away from $x^{vw}$ and towards $x^*$. 
 
 Step 4: For each vertex of the form $v^{vw}$ in Step 1, and each $x^*\in N_D(w)\setminus N_D[v]$ include the edge $v^{vw}x^*$.

\begin{defi}For each $vw\in \vec{E}(D)$, we call the copy of $H_{vw}$ that lives inside of $\mathcal{W}(D)$ the \textbf{$\mathbf{vw}$-star} of $\mathcal{W}(D)$.

\end{defi}
\newpage
\begin{exam}
\label{exam:WD}
Some examples of the $\mathcal{W}(D)$ construction.
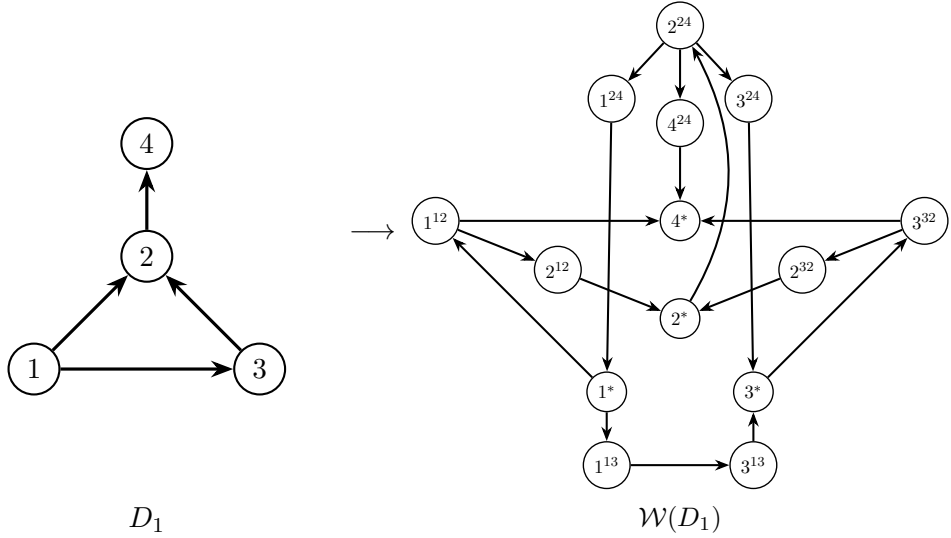
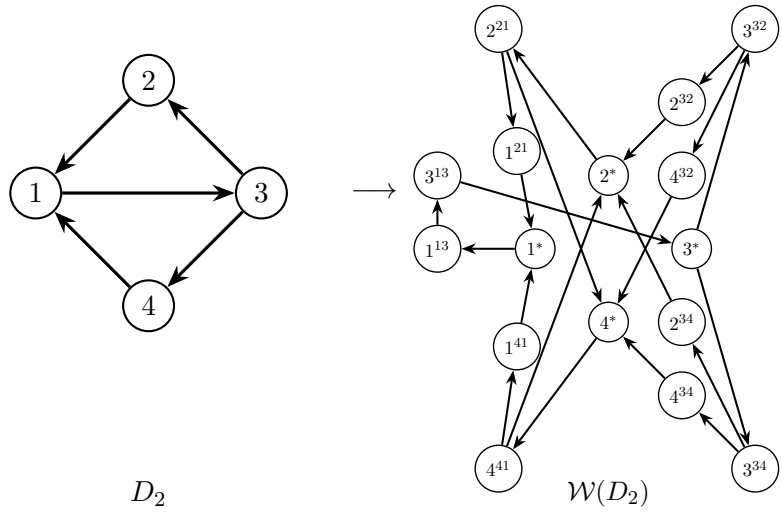
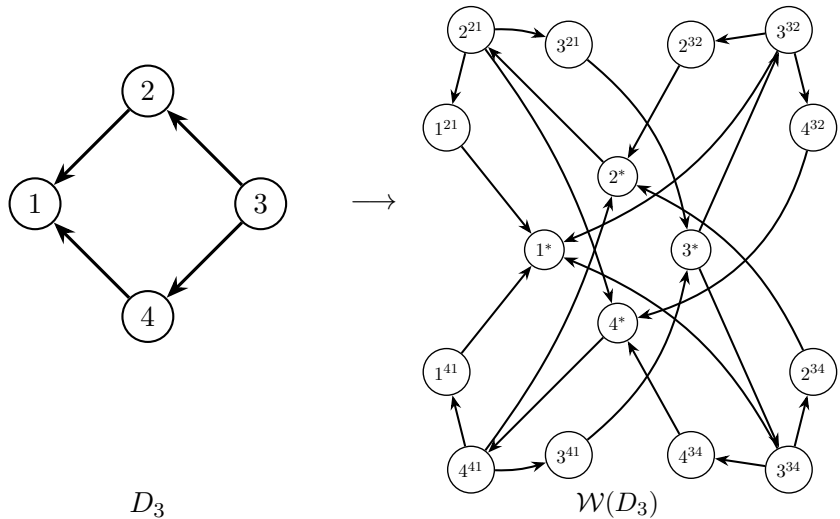
\begin{center}
\begin{tikzpicture}
\node (D_1) at (0,-2){$D_1$};
\node (nothing) at (3,1.8){$\longrightarrow$};
    \begin{scope}[every node/.style={circle,thick,draw}]
    \node (1) at (-1.5,0) {1};
    \node (2) at (0,1.5) {2};
    \node (3) at (1.5,0) {3};
    \node (4) at (0,3) {4};
   \end{scope}

\begin{scope}[>={Stealth[black]},
              every node/.style={fill=white,circle},
              every edge/.style={draw=black, very thick}]
   \path [->] (2) edge  (4);
   \path [->] (1) edge  (3);
   \path [->] (1) edge  (2);
      \path [->] (3) edge  (2);
\end{scope}
\end{tikzpicture}\scalebox{.65}{
\begin{tikzpicture}
\begin{scope}[every node/.style={circle,thick,draw}]

     \node (1) at (-1.5,0) {$1^*$};
    \node (2) at (0,1.5) {$2^*$};
    \node (3) at (1.5,0) {$3^*$};
    \node (4) at (0,3.5) {$4^*$};

    \node (112) at (-5,3.5) {$1^{12}$};
    \node (212) at (-2.5,2.5) {$2^{12}$};

   \node (113) at (-1.5,-1.5) {$1^{13}$};
    \node (313) at (1.5,-1.5) {$3^{13}$};
    
    \node (332) at (5,3.5) {$3^{32}$};
    \node (232) at (2.5,2.5) {$2^{32}$};
    
    
     \node (224) at (0,7.5) {$2^{24}$};
      \node (424) at (0,5.5) {$4^{24}$};
    \node (124) at (-1.4,6) {$1^{24}$};
    \node (324) at (1.4,6) {$3^{24}$};
    
\end{scope}

\begin{scope}[>={Stealth[black]},
              every node/.style={fill=white,circle},
              every edge/.style={draw=black, very thick}]
   \path [->] (112) edge  (212);
   \path [->] (112) edge  (4);
\end{scope}

\begin{scope}[>={Stealth[black]},
              every node/.style={fill=white,circle},
              every edge/.style={draw=black, very thick}]
   \path [->] (224) edge  (424);
   \path [->] (224) edge  (124);
   \path [->] (224) edge  (324);
\end{scope}

\begin{scope}[>={Stealth[black]},
              every node/.style={fill=white,circle},
              every edge/.style={draw=black, very thick}]
   \path [->] (113) edge  (313);
  
\end{scope}

   \begin{scope}[>={Stealth[black]},
              every node/.style={fill=white,circle},
              every edge/.style={draw=black, very thick}]
   \path [->] (332) edge  (232);
   \path [->] (332) edge  (4);
\end{scope}

\begin{scope}[>={Stealth[black]},
              every node/.style={fill=white,circle},
              every edge/.style={draw=black, very thick}]
\path [->] (1) edge  (112);
\path [->] (3) edge  (332);
\path [->] (1) edge  (113);
\path [->] (2) edge[bend right =30] (224);
\end{scope}

 \begin{scope}[>={Stealth[black]},
              every node/.style={fill=white,circle},
              every edge/.style={draw=black, very thick}]
\path [->] (212) edge  (2);
\path [->] (232) edge  (2);
\path [->] (313) edge  (3);
\path [->] (124) edge  (1);
\path [->] (324) edge  (3);
\path [->] (424) edge  (4);
\end{scope}
\node[scale=1.42] (WD) at (0, -2.6){$\mathcal{W}(D_1)$};
\end{tikzpicture}}

\vspace{.1cm}
\vspace{.1cm}
\begin{tikzpicture}
\node (D_2) at (0,-4){$D_2$};
\node (nothing) at (3,0){$\longrightarrow$};
    \begin{scope}[every node/.style={circle,thick,draw}]
    \node (1) at (-1.5,0) {1};
    \node (2) at (0,1.5) {2};
    \node (3) at (1.5,0) {3};
    \node (4) at (0,-1.5) {4};
   \end{scope}

\begin{scope}[>={Stealth[black]},
              every node/.style={fill=white,circle},
              every edge/.style={draw=black, very thick}]
   \path [->] (2) edge  (1);
   \path [->] (4) edge  (1);
   \path [->] (1) edge  (3);
   \path [->] (3) edge  (4);
    \path [->] (3) edge  (2);
    
\end{scope}
\end{tikzpicture}\scalebox{.65}{
\begin{tikzpicture}
\node (nothing) at (5, 0){};
\begin{scope}[every node/.style={circle,thick,draw}]

    \node (1) at (-1.5,0) {$1^*$};
    \node (2) at (0,1.5) {$2^*$};
    \node (3) at (1.7,0) {$3^*$};
    \node (4) at (0,-1.5) {$4^*$};

    \node (113) at (-3.5,0) {$1^{13}$};
    \node (313) at (-3.5,1.5) {$3^{13}$};
    
    \node (332) at (3,4.5) {$3^{32}$} ;
     \node (232) at (1.5,3) {$2^{32}$};
      \node (432) at (1.5,1.5) {$4^{32}$} ;
    
    \node (334) at (3,-4.5) {$3^{34}$} ;
    \node (234) at (1.5,-1.5) {$2^{34}$} ;
    \node (434) at (1.5,-3) {$4^{34}$} ;
   
    \node (221) at (-2.25,4.5) {$2^{21}$} ;
     \node (121) at (-1.875,2) { $1^{21}$};

    \node (441) at (-2.25,-4.5) {$4^{41}$};
    \node (141) at (-1.875,-2) {$1^{41}$};

\end{scope}

\begin{scope}[>={Stealth[black]},
              every node/.style={fill=white,circle},
              every edge/.style={draw=black, very thick}]
   
    \path [->] (221) edge  (121);
\end{scope}

\begin{scope}[>={Stealth[black]},
              every node/.style={fill=white,circle},
              every edge/.style={draw=black, very thick}]
   \path [->] (113) edge  (313);
    
\end{scope}


\begin{scope}[>={Stealth[black]},
              every node/.style={fill=white,circle},
              every edge/.style={draw=black, very thick}]
   \path [->] (441) edge  (141);

\end{scope}

\begin{scope}[>={Stealth[black]},
              every node/.style={fill=white,circle},
              every edge/.style={draw=black, very thick}]
   \path [->] (334) edge  (434);
   \path [->] (334) edge  (234);

\end{scope}
 \begin{scope}[>={Stealth[black]},
              every node/.style={fill=white,circle},
              every edge/.style={draw=black, very thick}]
   \path [->] (332) edge  (432);
   \path [->] (332) edge  (232);

\end{scope}
\begin{scope}[>={Stealth[black]},
              every node/.style={fill=white,circle},
              every edge/.style={draw=black,very thick}]
   \path [->] (1) edge  (113);
    \path [->] (4) edge  (441);
    \path [->] (3) edge  (334);
    \path [->] (3) edge  (332);
    \path [->] (2) edge  (221); 
\end{scope}

\begin{scope}[>={Stealth[black]},
              every node/.style={fill=white,circle},
              every edge/.style={draw=black,very thick}]
   \path [->] (434) edge  (4);
    \path [->] (234) edge (2);
    \path [->] (441) edge  (2);
    \path [->] (313) edge  (3); 
    \path [->] (121) edge  (1);
    \path [->] (221) edge  (4);
    \path [->] (232) edge  (2);
    \path [->] (432) edge  (4);
     \path [->] (141) edge  (1);
\end{scope}
\node[scale=1.42] (LD) at (0, -5){${\mathcal{W}(D_2)}$};
\end{tikzpicture}}
\vspace{.1cm}
\vspace{.1cm}
\begin{tikzpicture} 
\node (D_3) at (0,-4){$D_3$};
\node (nothing) at (3,0){$\longrightarrow$};
    \begin{scope}[every node/.style={circle,thick,draw}]
    \node (1) at (-1.5,0) {1};
    \node (2) at (0,1.5) {2};
    \node (3) at (1.5,0) {3};
    \node (4) at (0,-1.5) {4};
   \end{scope}

\begin{scope}[>={Stealth[black]},
              every node/.style={fill=white,circle},
              every edge/.style={draw=black, very thick}]
   \path [->] (2) edge  (1);
   \path [->] (3) edge  (2);
   \path [->] (3) edge  (4);
    \path [->] (4) edge  (1);
\end{scope}
\end{tikzpicture}
\scalebox{.65}{
\begin{tikzpicture}
\begin{scope}[every node/.style={circle,thick,draw}]

    \node (1) at (-1.5,0) {$1^*$};
    \node (2) at (0,1.5) {$2^*$};
    \node (3) at (1.5,0) {$3^*$};
    \node (4) at (0,-1.5) {$4^*$};

    \node (332) at (3.5,4.5) {$3^{32}$} ;
     \node (232) at (1.5,4.2) {$2^{32}$};
 \node (432) at (4,2.5) {$4^{32}$} ;
 
    \node (334) at (3.5,-4.5) {$3^{34}$} ;
    \node (434) at (1.5,-4.2) {$4^{34}$};
    \node (234) at (4,-2.5) {$2^{34}$} ;

    \node (221) at (-3,4.5) {$2^{21}$} ;
    \node (121) at (-3.5,2.5) { $1^{21}$};
    \node (321) at (-1,4.2) { $3^{21}$};

    \node (441) at (-3,-4.5) {$4^{41}$} ;
    \node (141) at (-3.5,-2.5) { $1^{41}$};
    \node (341) at (-1,-4.2) { $3^{41}$};

\end{scope}

\begin{scope}[>={Stealth[black]},
              every node/.style={fill=white,circle},
              every edge/.style={draw=black, very thick}]
    
    \path [->] (221) edge  (121);
     \path [->] (221) edge[bend left =10] (321);
\end{scope}


\begin{scope}[>={Stealth[black]},
              every node/.style={fill=white,circle},
              every edge/.style={draw=black, very thick}]
   \path [->] (441)  edge  (141);
   \path [->] (441) edge[bend right =10]  (341);
    
\end{scope}

 \begin{scope}[>={Stealth[black]},
              every node/.style={fill=white,circle},
              every edge/.style={draw=black, very thick}]
   \path [->] (334) edge  (234);
   \path [->] (334) edge  (434);
\end{scope}

 \begin{scope}[>={Stealth[black]},
              every node/.style={fill=white,circle},
              every edge/.style={draw=black, very thick}]
   \path [->] (332) edge  (432);
   \path [->] (332) edge  (232);
\end{scope}
\begin{scope}[>={Stealth[black]},
              every node/.style={fill=white,circle},
              every edge/.style={draw=black,very thick}]
    \path [->] (4) edge  (441);
    \path [->] (3) edge  (334);
    \path [->] (3) edge  (332);
    \path [->] (2) edge  (221); 
\end{scope}

\begin{scope}[>={Stealth[black]},
              every node/.style={fill=white,circle},
              every edge/.style={draw=black,very thick}]
   \path [->] (434) edge  (4);
    \path [->] (234) edge[bend right =20] (2);
    \path [->] (441) edge[bend right =10](2);
    \path [->] (121) edge  (1);
    \path [->] (221) edge[bend left =10]   (4);
    \path [->] (232) edge  (2);
    \path [->] (432) edge[bend left =30] (4);
    \path [->] (141) edge  (1);
    \path [->] (321) edge[bend left =20]  (3);
    \path [->] (341) edge[bend right =20]  (3);
    \path [->] (332) edge[bend left=20] (1);
    \path [->] (334) edge[bend right =20]  (1);
\end{scope}
\node[scale=1.42] (LD) at (0, -5.2){$\mathcal{W}(D_3)$};
\end{tikzpicture}}

\end{center}

\end{exam}

The short directed paths defined next will be an important tool throughout the remainder of this work, as it turns out that the Eulerian subdigraphs of $\mathcal{W}(D)$ can always be ``pieced together" using these paths. 

\begin{defi}
\label{def:gammapaths}Let $vw\in \vec{E}(D)$. For each $x\in N_D[v]\triangle N_D(w)$, define the directed path $_{v^*}P^{vw}_{x^*}$ in $\mathcal{W}(D)$ as follows: 

\[
  {}_{v^*}P^{vw}_{x^*} =
  \begin{cases}
            v^*\rightarrow v^{vw}\rightarrow x^{vw}\rightarrow x^* & \quad \text{ if } x\in N_D[v]\setminus N_D(w) \\
            v^*\rightarrow v^{vw}\rightarrow x^*& \quad \text{ if } x\in N_D(w)\setminus N_D[v] \\
  \end{cases}.
\]

That is, ${}_{v^*}P^{vw}_{x^*}$ is the unique directed path that starts at $v^*$, travels through the $vw$-star (and no other stars), and ends at $x^*$. We call these paths \textbf{$\pmb{\gamma}$-paths}. 
\end{defi}

Notice that whenever $x\in N_D[v]\setminus N_D(w)$, the path ${}_{v^*}P^{vw}_{x^*}$ has length $3$, and whenever $x\in N_D(w)\setminus N_D[v]$, the path ${}_{v^*}P^{vw}_{x^*}$ has length $2$. The parity of the length of the $\gamma$-paths will be important for the results later on.  

\begin{prop}
\label{prop:disjointgammapaths}
Two $\gamma$-paths ${}_{u*}P^{uv}_{w^*}$ and ${}_{x*}P^{xy}_{z^*}$ are edge disjoint if and only if $uv\neq xy$. 
\end{prop}

\begin{proof}
Suppose $uv=xy$.  By definition, both ${}_{u^*}P^{uv}_{w^*}$ and ${}_{x*}P^{xy}_{z^*}$ contain the edge $u^*u^{uv}$, so the two paths are not edge-disjoint. \

Suppose $uv\neq xy$. By construction of $\mathcal{W}(D)$, the $uv$-star and the $xy$-star are vertex disjoint. By definition of ${}_{u^*}P^{uv}_{w^*}$, every edge of ${}_{u^*}P^{uv}_{w^*}$ has an endpoint in the $uv$-star, and no edge of ${}_{u^*}P^{uv}_{w^*}$ has an endpoint in the $xy$-star. Similarly, every edge of ${}_{x*}P^{xy}_{z^*}$ has an endpoint in the $xy$-star and no edge of ${}_{x*}P^{xy}_{z^*}$ has an endpoint in the $uv$-star. Thus, the two paths have no common edges; they are edge-disjoint.
\end{proof}

\begin{thm}
\label{thm:eulerianspanning}
 Let $S$ be a spanning subdigraph of $\mathcal{W}(D)$. Then $S$ is Eulerian if and only if the following properties hold:
 
 \begin{enumerate}
     \item  The edge set of $S$ is a (possibly empty) union of edge disjoint $\gamma$-paths.\vspace{.2cm}
     
     \item For each vertex of the form $v^* $ in $ V(D)$, $d^+_S(v^*)=d^-_S(v^*)$.
 \end{enumerate}\vspace{.2cm}

\end{thm}

\begin{proof}
Suppose that $S$ is Eulerian. Then for each vertex of the form $v^*$ in $S$ we must have $d^+_S(v^*)=d^-_S(v^*)$, by definition of Eulerian. Furthermore, if $v^*$ is such that $d^+(v^*)>0$, then since every outgoing edge from $v^*$ enters some $vw$-star, and every vertex inside  a $vw$-star has in-degree at most one, $S$ must contain a directed path from $v^*$, into the $vw$-star, and back out of the $vw$-star. Since the only edges that exit a $vw$-star exit to a vertex of the form $x^*$ with $x^*\neq v^*$, it follows that this path is a path of the form ${}_{v^*}P^{vw}_{x^*}$. Since the occurence of such a path in $S$ increases the in-degree of $x^*$ in $S$, $S$ must contain another $\gamma$-path that leaves $x^*$, travels into some $xy$-star, and exits the $xy$-star to some $z^*$, with $z^*\neq x^*$. We must continue this pattern until we have cycled back to $v^*$, so that $d^+_S(v^*)=d^-_S(v^*)$. Repeating this process gives a cycle decomposition of $S$, where each cycle is further decomposed into $\gamma$-paths. 

Now, for the sake of contradiction, suppose there are two distinct $\gamma$-paths,   ${}_{u*}P^{uv}_{w^*}$ and ${}_{x*}P^{xy}_{z^*}$ in this decomposition that are not edge-disjoint. Then by \Cref{prop:disjointgammapaths}, $uv=xy$.  Since the two paths are distinct, they contain distinct outgoing edges from $u^{uv}
$, so we have $d^+_S(u^{uv})>1$. But by construction, $d^-_{\mathcal{W}(D)}(u^{uv})=1$, so $d^-_S(u^{uv})\leq 1$, and this contradicts that $S$ is Eulerian. Hence, it must be the case that the edge set of $S$ is a union of edge-disjoint $\gamma$-paths. \vspace{.2cm}

To prove the converse, suppose that $\vec{E}(S)$ can be written as a union $\bigcup_{i=1}^n\vec{E}\left( {}_{v^*_i}P^{v_iw_i}_{x^*_i}\right)$, where the paths in the union are edge disjoint. Assume that for vertices of the form $v^*_i$ or $x^*_i$ in the above union, $d^+_S(v^*_i)=d^-_S(v^*_i)$ and $d^+_S(x^*_i)=d^-_S(x^*_i)$. \Cref{prop:disjointgammapaths} implies that edge-disjoint $\gamma$-paths can only meet at their endpoints, and their endpoints are the vertices of the form $v^*_i$ or $x^*_i$, so we need only verify that if $u $ is an internal vertex of ${}_{v^*_i}P^{v_iw_i}_{x^*_i}$ for some $i$, then $d^+_S(u)=d^-_S(u)$. Since each ${}_{v^*_i}P^{v_iw_i}_{x^*_i}$ is a directed path and the $\gamma$-paths are edge-disjoint, we have $d^+_S(u)=d^-_S(u)=1$. Thus, $S$ is Eulerian, as claimed.
\end{proof}

\section{An Alon-Tarsi Style Theorem for Additive List Colorings}

In this section, we prove our main result, \Cref{thm:newat}. The proofs in this section follow the same outline as the proof of \Cref{thm:at}, but with $\gamma$-paths playing the role of edges.\vspace{.2cm}

Let $G$ be a graph and let $D$ be an orientation of $G$. Identify the vertices of $G$ with the set $\{1,2,..,.n\}$. Note that if there is some $x\in \mathbb{N}^{|V(G)|}$  so that the polynomial   $$f(x_1,..,x_n)=\prod_{vw\in \vec{E}(D)}\left(\sum_{u \in N_D(w)} x_u - \sum_{u\in N_D(v)}x_u\right)$$ is nonzero at $x$, then $x$ corresponds to an additive coloring of $G$.\vspace{.2cm}

Note also that after cancellations within each factor, the polynomial $f$ can be written as 
$$f(x_1,..,x_n)=\prod_{vw\in \vec{E}(D)}\left(\sum_{u \in N_D(w)\setminus N_D(v)} x_u - \sum_{u\in N_D(v)\setminus N_D(w)}x_u\right).$$

We relate this polynomial to $\mathcal{W}(D)$ in the next theorem.
\begin{thm}
\label{thm:big}
Let $D=(V,\vec{E})$ be a digraph, and identify the vertices of $V$ with the set $\{1,2,...,n\}$.  For each $v\in V$, let $d_v=d^+_D(v)$. For each $v\in V$ define a variable $x_v$, and define $f\in \mathbb{R}[x_1,..,x_n]$ by $$f(x_1,...,x_n)=\prod_{vw\in \vec{E}}\left(\sum_{u \in N_D(w)\setminus N_D(v)} x_u - \sum_{u\in N_D(v)\setminus N_D(w)}x_u\right).$$

  Let $M_f=\prod_{v\in V(D)}x_v^{d_v}$. Then there is a one-to-one correspondence between the spanning Eulerian subdigraphs of $\mathcal{W}(D)$ and the occurrences of $\pm {M_f}$ in the expansion of $f$, such that each occurrence of $-M_f$ in the expansion corresponds to an odd Eulerian subdigraph of $\mathcal{W}(D)$, and each occurrence of $+M_f$ corresponds to an even Eulerian subdigraph of $\mathcal{W}(D)$. 
\end{thm}
 
\begin{proof}
 For the sake of the proof, we will need to distinguish between the occurrences of a variable $x_v$ inside of a $vw$ factor of $f$ and the occurences $x_v$ inside of other factors. To do this, we temporarily replace the occurences of $x_v$ within a $vw$ factor with a new variable $y_v$; for each $v\in V(D)$, define two variables, $x_v$ and $y_v$, and define the polynomial $\hat{f}\in \mathbb{R}[x_1,...,x_n,y_1,...,y_n]$ by  $$\hat{f}(x_1,...,x_n,y_1,...,y_n)=\prod_{vw\in \vec{E}}\left(y_v+\sum_{u \in N_D(w)\setminus N_D[v]} x_u - \sum_{u\in N_D(v)\setminus N_D(w)}x_u\right).$$

Consider the monomials in the expansion of $\hat{f}$. (Here we are considering the monomials before ``collecting like terms," so there will be $\prod_{vw\in \vec{E}}|N_D(v)\triangle N_D(w) |$ such monomials.) Each monomial corresponds to picking one (signed) variable from each factor of $\hat{f}$. Therefore, if $M$ is such a monomial, we can write $M=\prod_{vw\in \vec{E}(D)}z_{vw}$ where each $z_{vw}$ is a (signed) variable that occurs in the $vw$ factor of $\hat{f}$. We see from the definition of $\hat{f}$, that for each $vw\in \vec{E}(D)$, either $z_{vw}=y_v$, or $z_{vw}=\pm x_u$ for some $u\neq v$. \vspace{.2cm}

Let $M$ be a monomial in the expansion of $\hat{f}$. Write $M=\prod_{vw\in \vec{E}(D)}z_{vw}$, as above. Associate to $M$ the subdigraph $S_M$ of $\mathcal{W}(D)$ defined by: $$V(S_M)=V(\mathcal{W}(D))$$
$$\vec{E}(S_M)=\bigcup\{\vec{E}({}_{v^*}P^{vw}_{u*}):z_{vw}=\pm x_u, u\neq v\}.$$
\vspace{.2cm}

 By \Cref{prop:disjointgammapaths}, the subdigraphs $S_M$ are precisely the spanning subdigraphs of $\mathcal{W}(D)$ whose edge set is a union of edge disjoint $\gamma$-paths. Furthermore, for distinct monomials $M$ and $M'$, $S_M\neq S_{M'}$. Thus, we have the following observation:\vspace{.2cm}

\noindent \underline{Observation 1:}  The correspondence $M\leftrightarrow S_M$ is a bijective correspondence between the monomials in the expansion of $\hat{f}$ and the spanning subdigraphs of $\mathcal{W}(D)$ whose edge set is a union of edge disjoint $\gamma$-paths.\vspace{.2cm}

 Let $M$ be a monomial in the expansion of $\hat{f}$ and write $M=\prod_{vw\in\vec{E}(D)}z_{vw}$, as above. By the definition of $\hat{f}$, whenever $z_{vw}=+x_u$, we have $u\in N_D(w)\setminus N_D[v]$, so \Cref{def:gammapaths} tells us that $_{v^*}P^{vw}_{u*}$ is a path of length two and therefore has an even number of edges. If $z_{vw}=-x_u$, then $u\in N_D(v)\setminus N_D(w)$, so $_{v^*}P^{vw}_{u^*}$ is a path of length three, and therefore has an odd number of edges. Set $$t^-=\#\{vw\in \vec{E}(D):z_{vw}=-x_u \text{ for some } u\}.$$ Then we see that $S_M$ has an odd number of edges if and only if $t^-$ is odd, which is true if and only if the sign of $M$ is negative in the expansion of $\hat{f}$. We record this as observation 2:\vspace{.2cm}

\noindent \underline{Observation 2}: $S_M$ has an odd number of edges if and only if the sign of $M$ is negative in the expansion of $\hat{f}$.

\vspace{.2cm}

Fix  $u\in V(D)$.  Let $M$ be a monomial in the expansion of $\hat{f}$ and write $M=\prod_{vw\in \vec{E}(D)}z_{vw}$, as above. Define the following notation: 
\begin{center}
\begin{align*}
k_u&=\#\{uw\in \vec{E}(D):z_{uw}\neq y_u\}\\
t_u&=\#\{vw\in\vec{E}(D):z_{vw}=\pm x_u\}\\
\end{align*}
\end{center}

 By the construction of $S_M$, there are exactly $k_u$ $\gamma$-paths in $S_M$ whose initial vertex is $u^*$, and exactly $t_u$ $\gamma$-paths whose terminal vertex is $u^*$. Hence, $d^+_{S_M}(u^*)=k_u,$ and $d^-_{S_M}(u^*)=t_u$, so we have that $d^+_{S_M}(u^*)=d^-_{S_M}(u^*)$ if and only if $t_u=k_u$. Furthermore, note that $t_u=\text{deg}_M(x_u)$,   and since $y_u$ occurs in exactly $d_u$ factors of $\hat{f}$,  $\text{deg}_M(y_u)=d_u-k_u$.   Therefore, $$d^+_{S_M}(u^*)=d^-_{S_M}(u^*) \text{ if and only if }\text{deg}_M(y_u)=d_u-t_u=d_u-\text{deg}_M(x_u).$$ Thus, by \Cref{prop:disjointgammapaths} and Observation 1, we have the following:    \vspace{.2cm}

\noindent \underline{Observation 3:} $S$ is a spanning Eulerian subdigraph of $\mathcal{W}(D)$ if and only if $S=S_M$ for some $M$ with $\text{deg}_M(y_v)+\text{deg}_M(x_v)=d_v$ for every $v\in V(D)$.  \vspace{.2cm}

If we now let $y_v=x_v$ in $\hat{f}$ for each $v$, we get back the polynomial $f$, and the monomials that are in correspondence with the spanning Eulerian subdigraphs of $\mathcal{W}(D)$ are now exactly the monomials $M$ in the expansion of $f$ such that for each $v\in V(D)$, $\text{deg}_M(x_v)=d_v$; that is, the monomials in correspondence with the spanning Eulerian subdigraphs of $\mathcal{W}(D)$ are precisely the occurences of $\pm M_f$ in the expansion of $f$.  Since letting $y_v=x_v$ does not change the sign of any monomial in the expansion, observations 2 and 3 guarantee that the occurrences of $+M_f$ are in correspondence with even spanning Eulerian subdigraphs, and the occurences of $-M_f$ are in correspondence with odd spanning Eulerian subdigraphs, as desired.  
\end{proof}

\begin{thm}
\label{thm:newat}
Let $G$ be a graph and let $D$ be an orientation of $G$. For each $v\in V(G)$, let $L(v)$ be a set of $d^+_D(v)+1$ positive integers. If $EE(\mathcal{W}(D))\neq EO(\mathcal{W}(D))$, then there is an additive coloring $\ell:V\rightarrow \mathbb{N}$ of $G$ such that $\ell(v)\in L(v)$ for each $v\in V$.  
\end{thm}

\begin{proof}
Write $V(D)=\{1,2,...,n\}.$ By \Cref{thm:big}, the coefficient of  $M_f=\prod_{v\in V(D)}x^{d_D^+(v)}$ in $$f(x_1,...,x_n)=\prod_{vw\in \vec{E}}\left(\sum_{u \in N_D(w)} x_u - \sum_{u\in N_D(v)}x_u\right)$$ is equal to $EE(\mathcal{W}(D))-EO(\mathcal{W}(D))$. Hence, if $EE(\mathcal{W}(D))\neq EO(\mathcal{W}(D))$, then the coefficient of $M_f$ in $f$ is nonzero. Since $M_f$ has maximum degree in $f$, there is some $x\in \prod_{i=1}^{n}L(i)$ for which $f(x)$ is nonzero, by the combinatorial nullstellensatz. This corresponds to an additive coloring of $G$. 
\end{proof}
In \cite{ref4}, it was demonstrated that if $D$ is an orientation of $G$, then there is a bijection between the orientations of $G$  that have the same out degree sequence as $D$ and the spanning Eulerian subdigraphs of $D$. Applying this idea to $\mathcal{W}(D)$, we see that if the number of orientations of the underlying undirected graph of $\mathcal{W}(D)$ with the same out-degree sequence as $\mathcal{W}(D)$ is odd, it cannot be the case that $EE(\mathcal{W}(D))=EO(\mathcal{W}(D))$. This yields the following corollary.

\begin{cor}
Let $G$ be a graph and $D$ an orientation of $G$. For each $v\in V(G)$, let $L(v)$ be a list of $d^+_D(v)+1$, positive integers. Let $H$ be the underlying undirected graph of $\mathcal{W}(D)$.  If there are an odd number of orientations of $H$ with the same out-degree sequence as $\mathcal{W}(D)$, then $G$ can be additively colored by assigning to each $v\in V(G)$ some element of $L(v)$. 
\end{cor}

\section{An Application}

We now give an application of \Cref{thm:newat}.  The proof of \Cref{thm:simplicial sink} demonstrates a relationship between directed cycles in $\mathcal{W}(D)$ and certain closed walks in $G$ of the same parity. It also demonstrates that in some cases, we can find specific orientations $D$ of $G$ that obstruct cycles (in this case, odd cycles) in $\mathcal{W}(D)$. Note that a vertex $v$ of $G$ is said to be \textit{simplicial} if the neighbors of $v$ in $G$ form a clique. 

\begin{thm}
\label{thm:simplicial sink}
Let $G=(V,E)$ and let $D$ be an orientation of $G$. For each $v\in V$, let $L(v)$ be a list of $d^+_D(v)+1$ positive integers. If for every odd (undirected) cycle $C$ of $G$, there exists some $u\in V(C)$ such that $u$ is simplicial in $G$ and $d^+_D(u)=0$, then there is an additive coloring of $G$ that assigns to each $v\in V$ an element of $L(v)$.
\end{thm}

\begin{proof}
We claim that in this case, $EO(\mathcal{W}(D))=0$. For the sake of contradiction, suppose that $D$ is an orientation that satisfies the hypotheses of the theorem, and that $\mathcal{W}(D)$ has some spanning Eulerian subdigraph with an odd number of edges. Then $\mathcal{W}(D)$ contains an odd directed cycle, $C_{\mathcal{W}}$. Since $C_{\mathcal{W}}$ is Eulerian, $C_{\mathcal{W}}$ can be decomposed into a sequence of $\gamma$-paths, by \Cref{prop:disjointgammapaths} (here we are thinking of the $\gamma$-paths as edge paths):  $$C_{\mathcal{W}}={}_{v^*_1}P^{v_1w_1}_{v^*_2} {}_{v^*_2}P^{v_2w_2}_{v^*_3} \cdots {}_{v^*_{j-1}}P^{v_{j-1}w_{j-1}}_{v^*_{j}}{}_{v^*_j}P^{v_jw_j}_{v^*_1}.$$

From $C_{\mathcal{W}}$, we can define a closed walk in $G$. For each $1\leq i \leq j$, first define short walks from $v_i$ to $v_{i+1(\mod{j})}$ in $G$ for each $i$:\vspace{.2cm}

\[ {}_{v_i}Q^{v_{i}w_{i}}_{v_{i+1}}
  =
  \begin{cases}
           v_i\rightarrow v_{i+1} \hspace{1.36cm}\text{ if } \hspace{.5cm} {}_{v^*_i}P^{v_iw_i}_{v^*_{i+1} }= v^*_{i}\rightarrow v^{v_iw_i}_i\rightarrow v^{v_i w_i}_{i+1}\rightarrow v^*_{i+1} \\
             v_i\rightarrow w_{i}\rightarrow v_{i+1} \quad\text{ if } \hspace{.5cm}{}_{v^*_i}P^{v_iw_i}_{v^*_{i+1}}= v^*_{i}\rightarrow v^{v_iw_i}_i\rightarrow v^*_{i+1}  \\
  \end{cases}
\]

 \vspace{.2cm}

Notice that ${}_{v_i}Q^{v_{i}w_{i}}_{v_{i+1}}$ has the same parity as ${}_{v^*_i}P^{v_iw_i}_{v^*_{i+1} }$.\vspace{.2cm}

Now define the walk $Q$ in $G$ by
$$Q= {}_{v_1}Q^{v_1w_1}_{v_2} 
{}_{v_2}Q^{v_2w_2}_{v_3} \cdots {}_{v_j}Q^{v_jw_j}_{v_1}.$$

Evidently, $Q$ is a closed walk in $G$. Furthermore, $Q$ has odd length since the parity of ${}_{v_i}Q^{v_{i}w_{i}}_{v_{i+1}}$ is the same as that of ${v^*_i}P^{v_iw_i}_{v^*_{i+1}}$, for each $i$, and $C_{\mathcal{W}}$ has odd length. Since $Q$ has odd length, it
must contain an odd cycle, C. By assumption, there is some $u\in V(C)$ that is simplicial in $G$ and has $d^+_D(u)=0$. Since $d^+_D(u)=0,$ we also have $d^+_{\mathcal{W}(D)}(u^*)=0$, so $u^*$ does not lie on the directed cycle $C_{\mathcal{W}}$. Thus, from the construction of $Q$, we see that there must be some $ i\in \{1,2,..,j\}$ for which  $u=w_i$, $u\neq v_{i}$ and $u\neq v_{i+1}$; that is, the directed path  $$_{v^*_i}P^{v_iu}_{v^*_{i+1}}=v^*_i\rightarrow v^{v_i u}_i\rightarrow v^*_{i+1}$$ occurs in the decomposition of $C_{\mathcal{W}}$.  By definition of $_{v^*_i}P^{v_iu}_{v^*_{i+1}}$, we must have that $v_i,v_{i+1}\in N_G(u)$ and $v_{i+1}\in N_G(u)\setminus N_G[v_i]$. So $v_{i+1}$ cannot be a neighbor of $v_i$, contradicting the assumption that $u$ is simplicial. \vspace{.2cm}

Hence, it must be the case that $EO(\mathcal{W}(D))=0$. Since the empty Eulerian subdigraph of $\mathcal{W}(D)$ is even, we have $EE(\mathcal{W}(D))\geq 1$, and the claim follows from \Cref{thm:newat}. 
 \vspace{.2cm}

\end{proof}

\Cref{thm:simplicial sink} may, at first exposure, seem stronger than it actually is. Indeed, if a graph $G$ contains $K_4$ as a subgraph, there is no orientation of $G$ for which every odd cycle contains a simplicial vertex with out-degree $0$. Furthermore, suppose that $D$ is an orientation of $G$ satisfying the hypothesis of \Cref{thm:simplicial sink}. We can get a 3-coloring of $G$ in the following way: If $S\subseteq V(G)$ is a smallest set of simplicial vertices in $G$ such that for all $x\in S$, $d^+_D(x)=0$, and such that every odd cycle in $G$ contains some $x\in S$, then since $S$ is an independent set, we can color all the vertices in $S$ the same color.  Then since $G-S$ has no remaining odd cycles, it is bipartite, and we can color $G-S$ with two new colors to achieve a $3$-coloring of $G$.

Note also that if $G$ is a tripartite graph that admits an orientation $D$ such that one of the color classes of $G$ contains only simplicial vertices with out-degree $0$, then since every odd cycle must travel through a vertex in this color class, we see that $D$ satisfies the hypothesis of \Cref{thm:simplicial sink}. In light of these observations, the content of \Cref{thm:simplicial sink} can be rephrased as follows. 

\begin{thm}
\label{thm:tripartite}
If $G$ is a tripartite graph that admits an orientation $D$ such that one of the color classes of $G$ contains only simplicial vertices with out-degree $0$, and $L(v)$ is a list of positive integers of length $d^+_D(v)+1$ for each $v\in V(G)$, then $G$ can be additively colored by assigning to each $v\in V(G)$ an element of $L(v)$. 
\end{thm}
Note that \cref{thm:tripartite} implies that $\eta_{\ell}(G)\leq \Delta (G)+1$ for tripartite graphs such that one of the color classes of $G$ contains only simplicial vertices. Note also that the following result, originally proved in \cite{ref8}, is an immediate consequence of \Cref{thm:simplicial sink}.   
\begin{cor}
\label{cor:bipartite} Let $G$ be a bipartite graph and $D$ an orientation on $G$. If $L(v)$ is a list of $d^+_D(v)+1$ positive integers for each $v\in V(G)$, then there is an additive coloring of $G$ that assigns to each $v$ an element of $L(v)$.
\end{cor}

We give a simple example that shows how \Cref{thm:simplicial sink} might be applied to certain non-bipartite graphs. 

\begin{exam} Let $C_{2k}$ be the cycle of length $2k$ with vertex set  $v_1,..,v_{2k}$. Let $G_{2k}$ be the graph constructed by adding a vertex $u_i$, for each $i\in\{1,...,2k\}$, to $C_{2k}$, and adding the edges $\{u_i,v_i\}$ and $\{u_i,v_{i+1}\}$ for each $i\in \{1,2,..,2k-1\}$, and the edges $u_{2k}v_{2k}$ and $u_{2k}v_1$. Consider the orientation $D$ of $G_{2k}$ where the copy of $C_{2k}$ in $G_{2k}$ is oriented cyclically and for each $i\in \{1,2,...,2k\}$, every edge incident to $u_i$ is oriented towards $u_i$. Then $d^+_D(u_i)=0$ and $d^+_D(v_i)=3$ for each $i\in\{1,2,..,2k\}$. Since every odd cycle of $G$ must contain some $u_i$, and every $u_i$ is simplicial and has out degree $0$, \Cref{thm:simplicial sink} implies that if $L(u_i)$ is a list containing only one positive integer for each $i\in\{1,2,...,2k\}$, and $L(v_i)$ is a list of $4$ positive integers for each $i\in \{1,2,...,2k\}$, then $G$ can be additively colored by assigning to each vertex an element from its respective list. The following graph is $G_6$ with the orientation described.
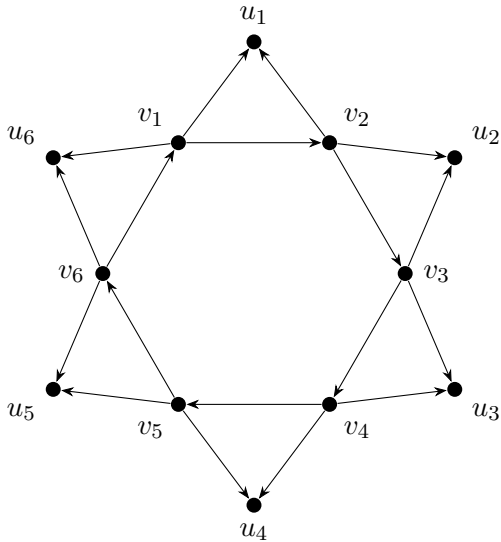
\begin{center}
\begin{tikzpicture}[scale=.67, >={Stealth[bend]},dot/.style={circle,fill,inner sep=2pt},
    declare function={R=3;},bend angle=12]
 \path[dash pattern=on 1.5pt off 1pt] 
 foreach \X [count=\Y] in {1,2,3,4,5,6}
 {(180-\Y*60:R) node[dot,label={180-\Y*60:$v_{\X}$}] 
 (v\X){}
 \ifnum\Y>1
  \ifnum\Y<7
   (v\the\numexpr\Y-1) edge[solid,->] (v\Y)
  \fi
 \fi
 };
 \path[->] 
(v6) edge (v1);
 \path[dash pattern=on 1.5pt off 1pt]
 foreach \X [count=\Y] in {1,2,3,4,5,6}
 {(150-\Y*60:4.6) node[dot,label={150-\Y*60:$u_{\X}$}] 
 (u\X){}
 \ifnum\Y>0
  \ifnum\Y<7
   (v\the\numexpr\Y) edge[solid,->] (u\Y)
  \fi
 \fi
 };
\path[->]
(v6) edge (u5)
(v5) edge (u4)
(v4) edge (u3)
(v3) edge (u2)
(v2) edge (u1)
(v1) edge (u6)
;
 
\end{tikzpicture}
\end{center}
\end{exam}

\section{Concluding Remarks}

 Note that, for any graph $G$, it is always the case that there exists an orientation $D$ of $G$ such that $EE(D)\neq EO(D)$. In particular, any acyclic orientation $D$ has 
$EE(D)=1$ and $EO(D)=0$. However, it is not necessarily the case that $\mathcal{W}(D)$ is acyclic whenever $D$ is acyclic. For example, the graphs $D_1$ and $D_3$ in \Cref{exam:WD}  are both acyclic, but $EE(\mathcal{W}(D_1))=3$ and $EO(\mathcal{W}(D_1))=1$, and $EE(\mathcal{W}(D_3))=12$ and $EO(\mathcal{W}(D_3))=0$ . (For the curious reader, $EE(\mathcal{W}(D_2))=2$ and $EO(\mathcal{W}(D_2))=8$.) 

Though \Cref{thm:simplicial sink} is an application of \Cref{thm:newat}, it is still an edge case, in the sense that $EO(\mathcal{W}(D))=0$. One can find a great deal of empirical evidence that certain classes of graphs equipped with specific orientations always yield $EE(\mathcal{W}(D))\neq EO(\mathcal{W}(D))$, but proving that this is the case appears to be difficult. 
We pose the following general question: For which graphs $G$ does there exist an orientation $D$ of $G$ so that $EE(\mathcal{W}(D))\neq EO(\mathcal{W}(D))$? We conjecture that such an orientation exists for every graph $G$.

\begin{conj}
\label{conj:conjecture}
For every graph $G$, there exists an orientation $D$ of $G$ such that $EE(\mathcal{W}(D))\neq EO(\mathcal{W}(D))$. 
\end{conj}

 Recall that the additive chromatic number of $G$ is denoted by $\eta(G)$ and the additive list chromatic number by $\eta_\ell(G)$ (see section 1 for definitions). The truth of \Cref{conj:conjecture} would imply that $\eta_l(G)\leq \Delta(G)+1$ for every graph $G$, where $\Delta(G)$ is the maximum degree of vertices of $G$. This bound would be a significant improvement on bounds proved in \cite{ref2},\cite{ref5}, and \cite{ref9}.
  
  Notably, Czerwiński et al. conjectured in \cite{ref8} that $\eta(G)\leq \chi (G)$ for all graphs $G$, where $\chi (G)$ is the classical chromatic number. The conjecture remains open, but some progress has been made--  for example, it was shown in \cite{ref6} that if $G$ is planar, $\eta(G)\leq 468$, and more results on the additive chromatic number of planar graphs can be found in \cite{ref1},\cite{ref7}, and \cite{ref9}. Progress for other types of graphs can be found in \cite{ref10}, where  an up-to-date list of classes of graphs for which the conjecture has been verified is given. Since $\eta(G)\leq \eta_{\ell}(G)$ for every graph $G$, \Cref{thm:newat} offers a possible method for finding upper bounds on $\eta(G)$, and therefore may be a relevant tool for helping to prove Czerwiński et al.'s conjecture for certain classes of graphs. 

\section*{Acknowledgements}

The author would like to thank Karen L. Collins for many insightful comments and suggestions, and Tao Wang for a suggestion that helped simplify the construction in Section 3.

\end{document}